\documentclass[11pt]{article}
\usepackage{geometry}
\usepackage[utf8]{inputenc}

\usepackage{fullpage}

\usepackage[pdftex]{graphicx}

\usepackage{wrapfig}
\usepackage{amsmath}
\usepackage{amsfonts}
\usepackage{amssymb}
\usepackage{amsthm}
\usepackage{bm}

\usepackage{tikz} 
\usetikzlibrary{patterns}
\usepackage{subcaption}

\usepackage[ruled,noresetcount]{algorithm2e}

\usepackage{hyperref}
\usepackage{cleveref}

\usepackage{mathtools}
\usepackage{multirow}

\usepackage{environ}

\usepackage{todonotes}

\usepackage{enumitem}

\newtheorem{theorem}{Theorem}

\newtheorem{lemma}{Lemma}
\newtheorem{example}{Example}

\theoremstyle{remark}
\newtheorem*{remark}{Remark}

\DeclareMathOperator*{\E}{E}

\DeclareMathOperator*{\Var}{Var}
\DeclareMathOperator*{\esssup}{ess\,sup}
\DeclareMathOperator*{\essinf}{ess\,inf}

\newcommand{\C}{\mathbb C}
\newcommand{\N}{\mathbb N}
\newcommand{\R}{\mathbb R}

\newcommand{\Z}{\mathbb Z}

\newcommand{\eps}{\varepsilon}

\newcommand{\abs}[1]{\left| {#1} \right|}

\newcommand{\ep}[2][]{\E_{{#1}}\left[ {#2} \right]}

\newcommand{\prb}[2][]{\Pr_{#1}\left[ #2 \right]}

\newcommand{\set}[1]{\left\{ {#1} \right\}}

\DeclarePairedDelimiterX{\infdivx}[2]{(}{)}{{#1} \,\delimsize\|\, {#2}}

\NewEnviron{splitEquation}{\begin{equation}\begin{split} \BODY \end{split}\end{equation}}
\NewEnviron{splitEquation*}{\begin{equation*}\begin{split} \BODY \end{split}\end{equation*}}

\makeatletter
\providecommand*{\cupdot}{%
  \mathbin{%
    \mathpalette\@cupdot{}%
  }%
}
\newcommand*{\@cupdot}[2]{%
  \ooalign{%
    $\m@th#1\cup$\cr
    \hidewidth$\m@th#1\cdot$\hidewidth
  }%
}
\providecommand*{\bigcupdot}{%
  \mathop{%
    \mathpalette\@bigcupdot{}%
  }%
}
\newcommand*{\@bigcupdot}[2]{%
  \ooalign{%
    $\m@th#1\bigcup$\cr
    \hidewidth$\m@th#1\cdot$\hidewidth
  }%
}
\makeatother

\makeatletter
\newcommand{\bigmodeproduct}[1]{
  \mathop{
    \mathchoice{\vcenter{\hbox{{\huge $\m@th\mkern-2mu\times\mkern-2mu$}$_{#1}$}}}          
               {\vcenter{\hbox{{\LARGE $\m@th\mkern-2mu\times\mkern-2mu$}$_{#1}$}}}         
               {\vcenter{\hbox{{$\m@th\mkern-2mu\times\mkern-2mu$}$_{#1}$}}}                
               {\vcenter{\hbox{{\footnotesize $\m@th\mkern-2mu\times\mkern-2mu$}$_{#1}$}}}  
  }\displaylimits
}
\makeatother

\title{On Sums of Monotone Random Integer Variables}
\author{Anders Aamand\thanks{Basic Algorithms Research Copenhagen (BARC), University of Copenhagen, Denmark. Emails: {\tt aa@di.ku.dk}, {\tt jakn@di.ku.dk}, and {\tt mikkel2thorup@gmail.com}. BARC is supported by the VILLUM Foundation grant 16582.} \and Noga Alon\footnote{Department of Mathematics, Princeton University, Princeton,
New Jersey, USA and Schools of Mathematics and Computer Science,
Tel Aviv University, Tel Aviv, Israel.
Email: {\tt nalon@math.princeton.edu}.
Research supported in part by
NSF grant DMS-1855464, BSF grant 2018267
and the Simons Foundation.} \and Jakob B\ae k Tejs Knudsen$^*$ \and Mikkel Thorup$^*$}

\begin{document}

\maketitle
\begin{abstract}
We say that a random integer variable $X$ is \emph{monotone} if the modulus of the characteristic function of $X$ is decreasing on $[0,\pi]$. This is the case for many commonly encountered variables, e.g., Bernoulli, Poisson and geometric random variables. In this note, we provide estimates for the probability that the sum of independent monotone integer variables attains precisely a specific value. We do not assume that the variables are identically distributed.  Our estimates are sharp when the specific value is close to the mean, but they are not useful further out in the tail. By combining with the trick of \emph{exponential tilting}, we obtain sharp estimates for the point probabilities in the tail under a slightly stronger assumption on the random integer variables which we call \emph{strong monotonicity}.
\end{abstract}
\section{Introduction}
In this note we provide sharp estimates for the probability that the sum of independent (not necessarily identically distributed) integer-valued random variables attains precisely a specific value. 
Our estimates hold under a fairly general assumption on the properties of the random variables, which for example is satisfied for Bernoulli, Poisson and geometric random variables. We proceed to describe these properties now.

Recall that for a real random variable $X$, the \emph{characteristic function} of $X$ is the map $f_X:\R \to \C$ given by $f_X(\lambda)=\E[e^{i\lambda X}]$. We say that a real random variable $X$ is \emph{monotone} if $|f_X|$ is decreasing on $[0,\pi]$. 
In the first part of this note (\Cref{sec:near-mean}), we provide estimates for the point probabilities of a sum of independent monotone random integer variables, $X=\sum_{j\in [k]}X_j$. To be precise, for any given $t\in \Z$, we estimate the probability $\Pr[X=t]$. Our estimates are sharp whenever $t$ is close to the mean $\E[X]$, but they are not useful further out in the tail. To handle point probabilities in the tail, we require a slightly stronger assumption on the random variables which we now describe.  

For a random integer variable  $X$ we define $I_X=\{\theta\in \R: \E[e^{\theta X}]<\infty\}$, to consist of those $\theta\in \R$ for which the moment generating function of $X$ is defined. We note that $I_X$ is an interval with $0\in I_X$. For $\theta \in I_X$, we may define the \emph{exponentially tilted} random variable $X_\theta$ by $\Pr[X_\theta=t]=\frac{\Pr[X=t]e^{\theta t}}{\E[e^{\theta X}]}$ for $t\in \Z$. We say that $X$ is \emph{strongly monotone} if (1) $I_X\neq \{0\}$ and (2) $X_\theta$ is monotone for each $\theta \in I_X$. In the second part of this note (\Cref{sec:in-tail}), we use the trick of exponential tilting to provide estimates for the point probabilities of a sum of independent strongly monotone random integer variables, $X=\sum_{j\in [k]}X_j$, which are also sharp in the tail.

It follows by direct computation that Bernoulli, Poisson, and geometric random variables are monotone, and moreover, that exponentially tilting these variables again yields Bernoulli, Poisson and geometric variables. In particular, these variables are all strongly monotone, so our results give sharp estimates for the point probabilities of the sum of (a mix of) such variables. In~\Cref{sec:in-tail}, we provide examples of the estimates that can be obtained for such a sum using our results. 

In the note we will consider the following setting.
Let $k$ be an integer and $(X_j)_{j\in [k]}$ independent integer-valued random variables with $\ep{X_j} = \mu_j$ and $\Var[X_j]=\sigma_j^2$ for $j\in [k]$\footnote{We define $[k] = \set{0, \ldots, k - 1}$}.
Let $X = \sum_{i \in [k]} X_i$, and further $\mu = \sum_{j \in [k]} \mu_j$ and $\sigma^2 = \sum_{j \in [k]}\sigma_j^2$ be respectively the expectation and variance of $X$.
The main result of the note is the following theorem.
\begin{theorem}\label{lem:density-geometric}
    There exists a universal constant $c$, such that if $X$ is monotone, 
    then for every $t$ for which
    $\mu + t \sigma$ is an integer, the probability that $X$ is precisely
    $\mu + t \sigma$ satisfies,
    \begin{align}
            \abs{\prb{X = \mu + t \sigma} - \frac{1}{\sqrt{2 \pi} \sigma}e^{-t^2/2}}
               \leq c\left(\frac{\sum_{j \in [k]} \ep{\abs{X_j - \mu_j}^3}}{\sigma^3}   \right)^2.\label{eq:main-estimate}         
    \end{align}
\end{theorem}
\begin{remark}
    We note that if each $X_j$ is monotone, then $X$ is as well.
    Indeed, the characteristic function of $X$ can be factorized as $f_X(\lambda)=\prod_{j\in [k]}f_{X_j}(\lambda)$.
    In particular,~\Cref{lem:density-geometric} holds when each of the variables $(X_j)_{j\in[k]}$ is monotone.
\end{remark}

Our result is reminiscent of the Berry-Esseen theorem, but instead of bounding the distance between the cumulative function of $X$ and the cumulative function of the normal distribution as the Berry-Esseen theorem does,
our result bounds the distance between the density function of $X$ and the density function of the normal distribution.
This setting has been studied before in the context of large deviation theory, e.g., by Blackwell and Hodges~\cite{Blackwell1959} and by Iltis~\cite{iltis1995sharp} in the $d$-dimensional case.
They do not require $X$ to be monotone but they only consider the case where $(X_j)_{j\in [k]}$ are identically distributed and are interested in the asymptotical behavior when $k\to \infty$.
In particular the distribution of the variables $(X_j)_{j\in [k]}$ cannot depend on $k$.
In this work we are not interested in such asymptotic bounds and our result is a uniform bound for monotone variables.

\section{Point Probabilities Near the Mean}\label{sec:near-mean}
The goal of this section is to prove \Cref{lem:density-geometric}, but before diving into the proof, we provide some examples of random variables for which the condition of the lemma is satisfied.
Let $p\in [0,1]$ and $\lambda>0$.
Let $Y$ be a Bernoulli variable with $\Pr[Y=1]=1-\Pr[Y=0]=p$, let $Z$ be geometric with $\Pr[Z=k]=p^k(1-p)$ for $k\in \N_0$, and let $W$ be Poisson with $\Pr[W=k]=\lambda^ke^{-\lambda}/k!$ for $k\in \N_0$.
Let $f_Y$, $f_Z$, $f_W$ be the characteristic functions for $Y$, $Z$, and $W$.
Then for $\lambda\in \R$,
\begin{align*}
f_Y(\lambda)=1-p+pe^{i\lambda}, \ \ f_Z(\lambda)=\frac{1-p}{1-pe^{i\lambda}}, \ \  \text{ and } \ \ f_W(\lambda)=e^{\lambda(e^{i\lambda}-1)}.
\end{align*}
Thus, 
\begin{align*}
|f_Y(\lambda)|^2&=(1-p+pe^{i\lambda})(1-p+pe^{-i\lambda})=1+2p(1-p)(\cos \lambda -1) \\
|f_Z(\lambda)|^2&=\left(\frac{1-p}{1-pe^{i\lambda}}\right)\left(\frac{1-p}{1-pe^{-i\lambda}}\right)=\frac{(1-p)^2}{1+p^2-2p\cos \lambda}, \quad \text{and} \\
|f_W(\lambda)|&=e^{\lambda(\cos \lambda-1)},
\end{align*}
which are all decreasing functions on $[0,\pi]$. 

We also need the following simple result concerning random integer variables.
\begin{lemma}\label{prop:thirdmoment}
Let $Y$ be an integer random variable $Y$ with third moment. Then 
$$
\E[|Y-\E[Y]|^3]\geq \Var[Y]/10.
$$
\end{lemma}
\begin{proof}
We may clearly assume that $-1/2<\E[Y]\leq 1/2$ by replacing $Y$ wit $Y-a$ for an appropriate integer $a$. We only consider the case $0\leq \E[Y]\leq 1/2$ as the other case follows from a symmetric argument. Let $p_k=\Pr[y=k]$ for $k\in \Z$ and $\mu=\E[Y]$. Then $\sum_{k=1}^\infty p_k k\geq \mu \geq \mu^2$. Moreover, for $k<0$, it holds that $|k-\mu|^3\geq (k-\mu)^2$ and for $k>0$, $|k-\mu|^3\geq \frac{1}{2}(k-\mu)^2\geq \frac{1}{8}k$. Letting $\lambda=1/5$, it follows that  
\begin{align*}
	\E[|Y-\mu|^3]\geq& \sum_{k<0}p_k(k-\mu)^2+\frac{\lambda}{2} \sum_{k>0}p_k(k-\mu)^2+\frac{1-\lambda}{8}\sum_{k>0}p_kk\\
	\geq& \frac{\lambda}{2}\sum_{k\in \Z\setminus \{0\}}p_k(k-\mu)^2+\frac{1-\lambda}{8}p_0\mu^2=\Var[Y]/10.
	\end{align*}
\end{proof}

\begin{proof}[Proof of \Cref{lem:density-geometric}]
We start by noting that we may assume that $\sigma^2>C$ for a sufficiently large constant $C$. Indeed, by~\Cref{prop:thirdmoment}, we have that $\frac{1}{\sigma^4}\sum_{j\in [k]}\E[|X_j-\mu_j|^3]\geq \frac{1}{10\sigma^2}$, so the result is trivial when $\sigma^2\leq C$ (by choosing $c$ sufficiently large).

    Now the proof proceeds, similarly to proofs of the Berry-Esseen theorem and uses simple properties of the Fourier transformation of $X$. Let $f_j$ be the characteristic function of $X_j-\mu_j$ for $j\in [k]$ and let $F$ be the characteristic function of $X-\mu$. Then  
    \[
        F(\lambda) = \prod_{j \in [k]} f_j(\lambda) = \sum_{n \in \Z} \prb{X = n}e^{i (n - \mu) \lambda} \; .
    \]
    For non-zero integers $s$, it holds that $\frac{1}{2\pi} \int_{-\pi}^\pi e^{i s \lambda} d\lambda=0$ whereas 
    $\frac{1}{2\pi} \int_{-\pi}^\pi e^{i s \lambda} d\lambda=1$ if $s=0$. It follows that for any integer $n \in \Z$,
    \[
        \prb{X = n} = \frac{1}{2\pi} \int_{-\pi}^\pi F(\lambda) e^{- i (n - \mu) \lambda} d\lambda \; .
    \]
    In particular, if $\mu + t \sigma$ is an integer, then
    \[
        \prb{X = \mu + t\sigma} = \frac{1}{2\pi} \int_{-\pi}^\pi F(\lambda) e^{- i t \sigma \lambda} d\lambda \; .
    \]

    We define $\tau = \frac{\sum_{j \in [k]} \ep{\abs{X_j - \mu_j}^3}}{\sigma^3}$ noting that we may assume that $\tau \le c_0$ for a sufficiently small constant $c_0$ as otherwise the result is trivial.
    Split the interval $[-\pi, \pi]$ into three parts, $I_1 = [-\eps, \eps]$, $I_2 = [\eps, \pi]$, and $I_3 = [-\pi, -\eps]$. We will prove that if $\eps=\frac{\sqrt{8\log 1/\tau}}{\sigma} $, then
    \begin{align}
        \abs{F(\lambda) e^{- i t \sigma \lambda}} &= \abs{F(\lambda)} \leq \tau^2    \quad \text{for all $\lambda \in I_2 \cup I_3$ and}, \label{eq:large-estimate} \\
        \int_{-\eps}^{\eps} F(\lambda) e^{- i t \sigma \lambda} d\lambda &= \frac{1}{\sqrt{2 \pi} \sigma}e^{-t^2/2} + O(\tau^2). \label{eq:Gaussian-estimate}
    \end{align}
    We note that we may assume that $\eps\leq \pi$. Indeed, by~\Cref{prop:thirdmoment}, $\frac{\log 1/\tau}{\sigma^2}\leq \frac{\log(8\sigma)}{\sigma}$, and $\sigma$ is assumed to be sufficiently large. The desired result thus follows immediately from \cref{eq:large-estimate} and~\eqref{eq:Gaussian-estimate}. 
    
    We start by proving~\eqref{eq:large-estimate}. It is a general fact that the characteristic function $f_Y$ of a random variable $Y$ is Hermitian, i.e., $f_Y(-t)=\overline{f_Y(t)}$. In particular,
    $\abs{F(\lambda)} = \abs{F(-\lambda)}$, so it is enough to prove that 
    $\abs{F(\lambda)} \le  \tau^2 $ for $\lambda \in I_2$.
    As $|F|$ is decreasing on $[0, \pi]$, it in fact suffices to prove that $\abs{F(\eps)} \le  \tau^2$.    
    Now another standard fact about the characteristic function $f_Y$ of a random variable $Y$ is that for any $n$,
    \begin{align}\label{eq:taylor}
    \left|f_Y(\lambda)-\sum_{j=0}^n \frac{(i\lambda)^j}{j!}\E[Y^j]\right|\leq \frac{|\lambda|^{n+1}\E[|Y|^{n+1}]}{(n+1)!}.
    \end{align}
By Jensen's inequality, $\sigma_j^2 \le (\E[|X_j-\mu_j|^3])^{2/3}$, so it follows that 
\begin{align}\label{eq:not-large}
\eps^2\sigma_j^2\le \left(\eps^3 \sum_{j \in [k]} \ep{\abs{X_j - \mu_j}^3}\right)^{2/3}=(\eps^3 \sigma^3 \tau)^{2/3}\le 8 \log(1/\tau) \tau^{2/3}\leq 1,
\end{align}
where the last inequality used that $\tau\leq c_0$ for a sufficiently small constant $c_0$.
We may thus apply~\eqref{eq:taylor} with $n=2$ to conclude that
    \begin{align*}
        \abs{f_j(\eps)} \leq 1 - \frac{\eps^2}{2} \sigma_j^2 + \ep{\abs{X_j - \mu_j}^3}\frac{\eps^3}{6} 
            \le e^{- \frac{\eps^2}{2} \sigma_j^2 +\ep{\abs{X_j - \mu_j}^3}\frac{\eps^3}{6}}.
    \end{align*}
    Thus, for $\lambda \in I_2$,
    \begin{align*}
        \abs{F(\lambda) e^{- i t \sigma \lambda}}
            = \abs{F(\lambda)}
            \le \abs{F(\eps)}
            \le e^{-\tfrac{\eps^2}{2} \sigma^2 + \left(\sum_{j \in [k]} \ep{\abs{X_j - \mu_j}^3} \right) \frac{\eps^3}{6}}=e^{-\eps^2\sigma^2(1/2-\sigma\eps\tau/6)}.
    \end{align*}
    As $\sigma \eps \tau=\tau\sqrt{8\log 1/\tau}\leq 3/2$, it therefore follows that for $\lambda \in I_2$,
    \begin{align}
        \abs{F(\lambda)}
            \le e^{-\frac{\eps^2\sigma^2}{4}}
            = \tau^2 \; ,
    \end{align}
    which proves \eqref{eq:large-estimate}.


    Turning to~\eqref{eq:Gaussian-estimate}, we again use the Taylor expansion formula to get 
    \begin{align*}
        f_j(\lambda) = 1 - \frac{\lambda^2}{2} \sigma_j^2 + \ep{\abs{X_j - \mu_j}^3} \lambda^3g_j(\lambda),
    \end{align*}
for some (complex-valued) function $g_j(\lambda)$ with $|g_j(\lambda)| \le 1/6$ for all $\lambda$.
As in~\eqref{eq:not-large}, for $\abs{\lambda} \le \eps$, 
    \begin{align}\label{eq:third-moment}
        \ep{\abs{X_j - \mu_j}^3} \abs{\lambda}^3
        \le 1,
    \end{align}
and 
\begin{align}\label{eq:comparing-terms}
\lambda^4\sigma_j^4 \le (|\lambda|^3\E[|X_j-\mu_j|^3])^{4/3}\leq |\lambda|^3\E[|X_j-\mu_j|^3]\leq 1.
\end{align}
It follows that $|f_j(\lambda)-1|\leq 5/6$. Now for $z\in \C$ with $|z|\leq 5/6$ it holds that $1+z=\exp(z+O(z^2))$. Also, if $a,b\in \C$ satisfy that $|a|^2\leq |b|\leq 1$, then $|a+b|^2\leq 2(|a|^2+|b|^2)\leq 4|b|$. Combining these observations with~\eqref{eq:third-moment} and~\eqref{eq:comparing-terms} we find that  
    \begin{align*}
        f_j(\lambda)= e^{- \frac{\lambda^2}{2} \sigma_j^2 + \ep{\abs{X_j - \mu_j}^3} O(\lambda^3)}.
    \end{align*}
    It follows that 
    \begin{align*}
        F(\lambda) e^{- i t \sigma \lambda}
            = e^{- \frac{\lambda^2}{2} \sigma^2 + \left( \sum_{j \in [k]} \ep{\abs{X_j - \mu_j}^3} \right) O(\lambda^3)} e^{- i t \sigma \lambda}
    \end{align*}
    We then get that 
    \begin{align}\begin{split}\label{eq:characteristic-Taylor}
        F(\lambda) e^{- i t \sigma \lambda}
            &= e^{- \frac{\lambda^2}{2} \sigma^2 + \left( \sum_{j \in [k]} \ep{\abs{X_j - \mu_j}^3} \right) O(\lambda^3)} e^{- i t \sigma \lambda}
            \\&= e^{- \frac{\lambda^2}{2} \sigma^2}\left(1 + \left( \sum_{j \in [k]} \ep{\abs{X_j - \mu_j}^3} \right) O(\lambda^3) \right) e^{- i t \sigma \lambda}
    \end{split}\end{align}
    for $\abs{\lambda} \le \eps$.
    Now we get that
    \begin{equation}\begin{split}\label{eq:Gaussian-third-moment}
        &\frac{1}{2\pi}\int_{-\eps}^{\eps} \abs{e^{- \frac{\lambda^2}{2} \sigma^2}  \left( \sum_{j \in [k]} \ep{\abs{X_j - \mu_j}^3} \right) O(\lambda^3) e^{- i t \sigma \lambda}} d \lambda
            \\&\quad\quad\quad\quad=  \frac{1}{\pi}\left( \sum_{j \in [k]} \ep{\abs{X_j - \mu_j}^3} \right) \int_{0}^{\eps} e^{- \frac{\lambda^2}{2} \sigma^2} O(\lambda^3) d \lambda
            \\&\quad\quad\quad\quad=  \frac{1}{\pi}\left( \frac{\sum_{j \in [k]} \ep{\abs{X_j - \mu_j}^3}}{\sigma^4} \right) \int_{0}^{\sqrt{8\log 1/\tau}} e^{- \frac{s^2}{2}} O(s^3) d s
            \\&\quad\quad\quad\quad=  O\left( \frac{\sum_{j \in [k]} \ep{\abs{X_j - \mu_j}^3}}{\sigma^4} \right)
            \\& \quad\quad\quad\quad =O(\tau^2)
    \end{split}\end{equation}
    Here we used the substitution $s = \lambda \sigma$, and the last step uses~\Cref{prop:thirdmoment}. Again using the same substitution we get that
    \begin{align*}
        \frac{1}{2\pi}\int_{-\eps}^{\eps} e^{- \frac{\lambda^2}{2} \sigma^2} e^{- i t \sigma \lambda} d \lambda
            = \frac{1}{2\pi \sigma}\int_{-\sqrt{8\log 1/\tau}}^{\sqrt{8\log 1/\tau}} e^{- \frac{s^2}{2}} e^{- i t s} d s
    \end{align*}
    We clearly have that 
    \begin{align}\label{eq:Gaussian-large}
        \frac{1}{2\pi \sigma} \int_{|s| \ge \sqrt{8\log 1/\tau}} \abs{e^{- \frac{s^2}{2}} e^{- i t s}} d s
            \le O\left(\frac{\tau}{\sigma} \right)
            \le O\left( \frac{\sum_{j \in [k]} \ep{\abs{X_j - \mu_j}^3}}{\sigma^4} \right)=O(\tau^2)
    \end{align}
    And calculating the Fourier transform of function of $e^{- \frac{s^2}{2}}$ we get that
    \begin{align}\label{eq:Gaussian}
        \frac{1}{2\pi \sigma} \int_{-\infty}^{\infty} e^{- \frac{s^2}{2}} e^{-i t s} ds
            = \frac{1}{\sqrt{2 \pi} \sigma } e^{-t^2/2}
    \end{align}
    Combining \eqref{eq:characteristic-Taylor}, \eqref{eq:Gaussian-third-moment}, \eqref{eq:Gaussian-large}, and \eqref{eq:Gaussian} proves \eqref{eq:Gaussian-estimate}.
    This finishes the proof.
\end{proof}

\section{Point Probabilities in the Tail}\label{sec:in-tail}
As is,~\Cref{lem:density-geometric} is only useful when $|t\sigma|$ is not too large. Indeed, for large $|t|$, the term $\frac{1}{\sqrt{2\pi}\sigma}e^{-t^2/2}$ will typically be much smaller than the error term on the right hand side of~\cref{eq:main-estimate}. We now show that if our variables satisfy the stronger property of being strongly monotone, we may also obtain precise estimates for the point probabilities in the tail by combining with the trick of exponential tilting. 

Recall that we defined a real random variable $X$ to be strongly monotone if $I_X\neq \{0\}$ and $X_\theta$ is monotone for each $\theta \in I_X$. Here, $I_X=\{\theta \in \R: \E[e^{\theta X}]<\infty\}$ consisted of those $\theta$ for which the moment generating function of $X$ is defined, and $X_\theta$ was the exponentially tilted random variable defined by $\Pr[X_\theta =t]=\frac{\Pr[X=t]e^{\theta t}}{\E[e^{\theta X}]}$ for $t\in \Z$.

Many commonly encountered random variables have the property of being strongly monotone:
\begin{lemma}\label{prop:strong-monotone-ex}
Let $X$ be Bernoulli, $Y$ be geometric and $Z$ be Poisson. Then $X,Y$ and $Z$ are each strongly monotone.
\end{lemma}
\begin{proof}
We already saw that the classes of Bernoulli, geometric, and Poisson variables consists of monotone variables. The result follows by  calculating the point probabilities of the tilted variables (when they exists) and observing that each class is closed under exponential tilts.
\end{proof}

Now suppose $X=\sum_{j\in [k]}X_j$ is a sum of independent random integer variables and moreover that $X$ is not almost surely equal to a constant. 
We are interested in estimates for the probability $\Pr[X=t]$ for some $t\in \Z$.
Let $I_j=\{\theta \in \R:\E[e^{\theta X_j}]<\infty\}$ and $I=\{\theta \in \R:\E[e^{\theta X}]<\infty\}=\cap_{j\in [k]}I_j$. We note each $I_j$ and $I$ are intervals containing $0$.
We define\footnote{Recall that the essential infimum and supremum of a random variable $X$ are defined by $\essinf X=\sup\{t: \Pr[X<t]=0\}$ and $\esssup X=\inf\{t: \Pr[X>t]=0$\} which are values in $\R \cup \{-\infty,\infty\}$.} $A=\essinf X$ and $B=\esssup X$. 
Let further $\psi_X:I\to \R$ be the cumulant generating function defined by $\psi_X:\theta \mapsto \log(\E[e^{\theta X}])$. It is well known that $\psi_X$ is strictly convex and infinitely often differentiable for $\theta$ lying in the interior of $I$ with $\psi_X'(\theta)=\frac{\E[Xe^{\theta X}]}{\E[e^{\theta X}]}$. For $t\in \R$, we define $g(t)=\sup_{\theta\in I}(\theta t-\psi_X(\theta))$. Now it is a standard fact about the cumulant generating function that if $I$ contains a non-empty open interval (i.e., consists of more than a single point), then $\inf_{\theta \in I}\psi_X'(\theta)=A$ and $\sup_{\theta \in I}\psi_X'(\theta)=B$. If in particular $A<t<B$, there exists a $\theta_0$ in the interior of $I$ with $\psi_X'(\theta_0)=t$. Moreover, this $\theta_0$ is unique since $\psi_X$ is strictly convex. 

Now let $(Y_j)_{j\in [k]}$ be independent random variables obtained by tilting each $X_j$ by $\theta_0$ as above. Let further $Y=\sum_{j\in [k]}Y_j$. For $s\in \Z$, we define $A_s=\{z\in \Z^k:z_1+\cdots+z_k=s\}$. Then for any $t\in \Z$,
$$
\Pr[X=t]=\sum_{z\in A_t}\prod_{j\in [k]} \Pr[X_j=z_j]=\frac{\E[e^{\theta_0 X}]}{e^{\theta_0 t}}\sum_{z\in A_t}\prod_{j\in [k]} \Pr[Y_j=z_j]=\frac{\E[e^{\theta_0 X}]}{e^{\theta_0 t}} \Pr[Y=t],
$$
so $Y$ is simply the variable obtained by tilting $X$ by $\theta_0$. 
Moreover, by the choice of $\theta_0$,
$$
\E[Y]=\sum_{z\in \Z}\frac{\Pr[X=z]e^{\theta_0 z}z}{\E[e^{\theta_0 X}]}=\frac{\E[Xe^{\theta_0 X}]}{\E[e^{\theta_0 X}]}=\psi_X'(\theta_0)=t.
$$
Now the fact that $\E[Y]=t$, suggests using~\Cref{lem:density-geometric} to estimate the probability that $\Pr[Y=t]$. Doing so, we immediately obtain the following result.
\begin{theorem}\label{lemma:tail-points}
Assume that $X$ is strongly monotone and not almost surely equal to a constant. Moreover assume that $I\neq \{0\}$. Let $t$ be an integer with with $A<t<B$ and $\theta$ be the unique real in the interior of $I$ having $\psi_X'(\theta)=t$. Let $Y$ be the exponential tilt of $X$ by $\theta$. Then $\E[Y]=t$ and 
\begin{align}\label{eq:tail-point-estimate}
\Pr[X=t]=\frac{\E[e^{\theta X}]}{e^{\theta t}} \left(\frac{1}{\sqrt{2\pi}\sigma_Y}\pm O\left(\frac{\eta_Y^2}{\sigma_Y^6}\right)\right),
\end{align}
where $\sigma_Y^2=\Var[Y]$ and $\eta_Y=\sum_{j\in [k]}\E[|Y_j-\E[Y_j]|^3]$.
\end{theorem}
\begin{remark} We note that if either $A=\essinf X\neq -\infty$ or $B=\esssup X\neq \infty$, then $[0,\infty)\subset I$ or $(-\infty, 0]\subset I$, respectively, and we can therefore always apply the exponential tilt in the lemma. We moreover, not that for $t<A$ and $t>B$, it trivially holds that $\Pr[X=t]=0$ and it is an easy exercise to show that 
\begin{align*}
\Pr[X=A]=\prod_{j\in [k]}\Pr[X_j=\essinf X_j], \quad \text{and} \quad \Pr[X=B]=\prod_{j\in [k]}\Pr[X_j=\esssup X_j],
\end{align*}
whenever $A\neq -\infty$ and $B\neq \infty$. Even though the lemma does not provide estimates for these probabilities, they are therefore usually easy to determine for concrete families of random variables.  
\end{remark}
To apply~\Cref{lemma:tail-points}, for $X=\sum_{j\in [k]}X_j$ a concrete sum of strongly monotone random variables, say geometric variables, we would calculate $\psi_X$ and find the unique $\theta$ with $\psi_X'(\theta)=t$. We would then determine the tilted random variables $(Y_j)_{j\in [k]}$. Typically $Y_j$ comes from the same family of random variables as $X_j$, e.g., an exponential tilt of respectively a Bernoulli, geometric, and Poisson variable is again Bernoulli, geometric and Poisson. We would then determine the quantities $\eta_Y$ and $\sigma_Y^2$ and plug into~\cref{eq:tail-point-estimate}.
\begin{example}
Let $X=\sum_{j\in [k]}X_j$, where $(X_j)_{j\in [k]}$ are independent Bernoulli variables with $\Pr[X_j=1]=p_j$. We want to estimate the $\Pr[X=t]$ for some $0< t <k-1$. For this, we define $\theta$, $(Y_j)_{j\in [k]}$ and $Y$ as in~\Cref{lemma:tail-points}. Then each $Y_j$ is again Bernoulli. If $\Pr[Y_j=1]=q_j$, then $\E[|Y_j-\E[Y_j]|^3]=q_j(1-q_j)(q_j^2+(1-q_j)^2)\leq \Var[Y_j]$, so that $\eta_Y\leq \sigma_Y^2$. Thus, the bound of~\eqref{eq:tail-point-estimate} becomes
\begin{align*}
\Pr[X=t]=\frac{1}{\sqrt{2\pi}\sigma_Y}\frac{\E[e^{\theta X}]}{e^{\theta t}} \left(1\pm O\left(\frac{1}{\sigma_Y}\right)\right),
\end{align*}
If $\sigma_Y$ is larger than a sufficiently large constant, the estimate is therefore asymptotically tight. Consider as a very simple example\footnote{This example is to be seen as a proof of concept as one can of course obtain precise asymptotically tight estimates directly using Stirling's formula.}, the case where the Bernoulli variables $(X_j)_{j\in [k]}$ are identically distributed. Then the same holds for the $(Y_j)_{j\in [k]}$, and since $\E[Y]=t$, we must have that $\sigma_Y^2=t(1-t/k)$. It follows that the bound above is asymptotically tight whenever $C<t<k-C$ for a sufficiently large universal constant $C$.
\end{example}
\begin{example}
Let $X=\sum_{j\in [k]}X_j$ be a sum of independent geometric variables such that for some probabilities $(p_j)_{j\in [k]}$ and each $s\in \N_0$, $\Pr[X_j=s]=p_j^s(1-p_j)$. Let $\mu_j=\E[X_j]=\frac{p_j}{1-p_j}$ for $j\in [k]$ and $\mu=\E[X]$. Assume that $\mu_j=O(1)$ for $j\in [k]$.
We want to estimate $\Pr[X=t]$ for some integer $t>0$ using~\Cref{lemma:tail-points}, and we define $\theta$, $(Y_j)_{j\in [k]}$ and $Y$ accordingly. For simplicity, we will assume that $t=O(\E[X])$. By~\Cref{prop:strong-monotone-ex}, each $Y_j$ is again geometric, say with $\Pr[Y_j=s]=q_j^s(1-q_j)$ for $s\in \N_0$. Moreover, since $\E[X_j]=O(1)$ for $j\in [k]$ and $t=O(\E[X])$, it follows that also $\E[Y_j]=O(1)$ for $j\in [k]$. Now simple calculations yields that $\Var[Y_j]=\Theta(\E[Y_j])$ and $\E[|Y_j-\E[Y_j]|^3]=\Theta(\E[Y_j])$. Plugging into~\cref{eq:tail-point-estimate}, we thus obtain that
\begin{align*}
\Pr[X=t]=\frac{1}{\sqrt{2\pi}\sigma_Y}\frac{\E[e^{\theta X}]}{e^{\theta t}} \left(1\pm O\left(\frac{1}{\sqrt{\mu_Y}}\right)\right),
\end{align*}
where $\mu_Y=\E[Y]=t$. In particular, the bound is asymptotically tight whenever $t\geq C$ for a sufficiently large constant $C$.
\end{example}



\bibliographystyle{acm}
\bibliography{fourier.bib}

\end{document}